\documentclass{amsart}
\usepackage{amsmath}
\usepackage{amssymb}
\usepackage{color}

\newtheorem{theorem}{Theorem}[section]
\newtheorem{lemma}[theorem]{Lemma}
\newtheorem{proposition}[theorem]{Proposition}

\theoremstyle{definition}
\newtheorem{definition}[theorem]{Definition}
\newtheorem{example}[theorem]{Example}

\theoremstyle{remark}
\newtheorem{remark}[theorem]{Remark}
\numberwithin{equation}{section}

\def\separa{\hbox to 14 truecm{\hrulefill}}

\author[P. Danchev]{Peter Danchev}
\address{Institute of Mathematics and Informatics, Bulgarian Academy of Sciences, 1113 Sofia, Bulgaria}
\email{danchev@math.bas.bg}
\thanks{The first author was partially supported by the Bulgarian National Science Fund under Grant KP-06 No. 32/1 of December 07, 2019.}

\author[E. Garc\'\i a]{Esther Garc\'\i a}
\address{ Departamento de Matem\'{a}tica  Aplicada, Ciencia e Ingenier\'{\i}a de los Materiales y Tecnolog\'{\i}a Electr\'onica,
Universidad Rey Juan Carlos, 28933 M\'{o}s\-to\-les (Madrid), Spain}
\thanks{The second author was partially supported by Ayuda Puente 2022, URJC}
\email{esther.garcia@urjc.es}

\author[M. G\'omez Lozano]{Miguel G\'omez Lozano}
\address{Departamento de \'Algebra, Geometr\'{\i}a y
Topolog\'{\i}a, Universidad de M\'alaga, 29071 M\'alaga, Spain}
\thanks{The three authors were partially supported by  the Junta de Andaluc\'{\i}a FQM264.}
\email{miggl@uma.es}

\begin{document}

\title[Decompositions of Matrices into a Sum of Unit and Nilpotent Matrices]{Decompositions of Matrices into a Sum of \\ Invertible Matrices and Matrices of \\ Fixed Nilpotence}
\maketitle

\begin{abstract} For any $n\ge 2$ and fixed $k\ge 1$, we give necessary and sufficient conditions for an arbitrary nonzero square matrix in the matrix ring $\mathbb{M}_n(\mathbb{F})$ to be written as a sum of an invertible matrix $U$ and a nilpotent matrix $N$ with $N^k=0$ over an arbitrary field $\mathbb{F}$.
\end{abstract}

%%%%%%%%%%%%%%%%%%%%%%%%%%%%%%%%%%%%%%%%%%%%%%%%%%%%%%%%%%%%%%%%%%%%%%%%%%%%%%%%%%%%%%%

%%%%%%%%%%%%%%%%%%%%%%%%%%%%%%%%%%%%%%%%%%%%%%%%%%%%%%%%%%%%%%%%%%%%%%%%%%%%%%%%%%%%%%%%

\section{Introduction}

In 1977, when studying lifting properties of idempotents, Nicholson defined an element $a$ in a ring $R$ to be {\it clean} if it can be written in the form $e + u$, where $e$ is an idempotent and $u$ is a unit (i.e., an invertible); see \cite{N}. If every element in a ring $R$ is clean, the ring is called clean. Inspired by these definitions, in 2013 Diesl \cite{Di} defined a ring element $b\in R$ to be {\it nil-clean} if it can be expressed as a sum of an idempotent and a nilpotent element, and the ring $R$ is nil-clean if every element in $R$ is so.

By combining the notions of invertibility and nilpotence, C\v{a}lug\v{a}reanu and  Lam introduced in 2016 the notion of {\it fine} rings: those in which every nonzero element can be written as the sum of an invertible element and a nilpotent one; see \cite{CL}. One of the main results of the paper \cite{CL} is the fact that every nonzero square matrix over a division ring is the sum of an invertible matrix and a nilpotent matrix. Indeed, they proved that as soon as the division ring has more than two elements, every nonzero square matrix over such division ring is similar to what they call a matrix in {\it good form}, i.e., a matrix with all diagonal entries nonzero. By decomposing this last matrix into its (invertible) lower part and its strictly (nilpotent) upper part, one concludes matrix rings over division rings with more than two elements are fine. Moreover, they separately proved that nonzero matrices over $\mathbb{F}_2$ are also clean, reaching to the desired result.

In the same paper (see the Acknowledgements section), the authors remarked that there was no previous reference to the fact that every square nonzero matrix (even over the complex field) could be expressed as the sum of a nilpotent matrix and an invertible one. Notice that the nilpotent matrices in C\v{a}lug\v{a}reanu and Lam decomposition have high indices of nilpotence because they correspond to the strictly  upper part of a matrix in good form.

The rings whose nonzero idempotents are fine turned out to be an interesting class of indecomposable rings and were studied in \cite{CZ1} by  C\v{a}lug\v{a}reanu and Zhou. In 2021, the same authors focused on rings in which every nonzero nilpotent element is fine, which they called $NF$ rings, and showed that for a commutative ring $R$ and $n\ge 2$, the matrix ring $\mathbb{M}_n(R)$ is $NF$ if and only if $R$ is a field; see \cite{CZ}.

On the same vein, a slightly more general class of rings than the aforementioned class of fine rings was defined in \cite{Da} under the name {\it nil-good rings} and some their characteristic properties, including the behaviour of the matrix ring over a nil-good ring, were explored in \cite{CMN} and \cite{GS}, respectively.

\medskip

In our work, we begin by fixing a bound  $k\ge 1$ for the index of nilpotence of the nilpotent part and pose the following problem for matrices over fields:

\medskip

{\bf Problem:} {\it Given $k\ge 1$, find necessary and sufficient conditions to decompose any nonzero square matrix $A$ over a field $\mathbb{F}$ as a sum of an invertible matrix $U$ and a nilpotent matrix $N$ with $N^k=0$.}

\begin{remark}\label{necessarycondition} Notice that the problem of decomposing a matrix as the sum of a unit matrix and a nilpotent matrix of index at most $k$ is not true in general. In fact, invertible square matrices have full rank, and the rank of a nilpotent matrix of index $k$ is the sum of the rank of every nilpotent block of index $k_i$ (whose rank is $k_i-1$) in the Jordan canonical form of $N$. Therefore, in the matrix ring $\mathbb{M}_n(\mathbb{F})$ over the field $\mathbb{F}$, if we decompose $n=ck+d$ where $c,d\in \mathbb{N}$ and $0\le d<k$, the rank of every nilpotent matrix is  less than or equal to
$$\left\{
\begin{array}{ll}
 c(k-1)=n-\frac nk, & \hbox{if $d=0$;}\\
c(k-1)+d-1=n-\frac nk+\frac dk-1\le n-\frac nk, & \hbox{if $d>0$}
\end{array}
\right.$$
($c$ blocks of index $k$, and one block of index $d$ when $d>0$, in its Jordan canonical form), so a necessary condition for this decomposition to hold is that the rank of the original matrix must be greater than or equal to $\frac nk$. To illustrate this more concretely, let $k\ge 2$, suppose $n\ge k+2$ and let $A=e_{12}\in \mathbb{M}_{n}(\mathbb{F})$ be the standard matrix. If we assume in a way of contradiction that $A=U+N$, where $U$ is an invertible matrix and $N^k=0$, then one may write that $U=A-N$. But the rank of an invertible matrix is always maximal (that is, exactly $n$ in this case), whereas the rank of $A$ is one and the rank of $N$ is $\le n-\frac nk$, so it cannot be recovered a rank $n$ matrix from a matrix of rank $1$ and a matrix of rank at most $n-\frac nk$.
\end{remark}

In this paper we completely solve this problem for matrices over arbitrary fields, proving that following result:

\medskip

{\bf Theorem.} {\it Let $\mathbb{F}$ be a field, let $n\ge 2$, and let us fix $k\ge 1$. Given a nonzero matrix $A\in \mathbb{M}_{n}(\mathbb{F})$, there exists an invertible matrix $U\in \mathbb{M}_{n}(\mathbb{F})$ and a nilpotent matrix $N\in \mathbb{M}_{n}(\mathbb{F})$ with $N^k=0$ such that $A=U+N$ if, and only if, the rank of $A$ is greater than or equal to $\frac{n}{k}$.}

\medskip

Since the properties invertibility and nilpotence are both invariant conditions under similarity, we will use the primary rational canonical form of a matrix (\cite[VII.Corollary 4.7(ii)]{H}), which states that every matrix $A\in \mathbb{M}_n(\mathbb{F})$, where $\mathbb{F}$ is a field, is similar to a direct sum of companion matrices of prime power polynomials $p_1^{m_{11}}, \dots, p_s^{m_{sk_s}}\in \mathbb{F}[x]$ where each $p_i$ is prime (irreducible) in $\mathbb{F}[x]$. The matrix $A$ is uniquely determined except for the order of these companion matrices. The polynomials $p_1^{m_{11}}, \dots, p_s^{m_{sk_s}}$ are called {\it the elementary divisors} of the matrix $A$.

\section{Decomposing Matrices into a Sum of Invertible and Nilpotent Matrices}

In our argument we will separate the elementary divisors $q(x)$ of a matrix $A\in \mathbb{M}_{n}(\mathbb{F})$ between those  that satisfy  $q(0)\ne 0$ and those with $q(0)=0$, i.e., $q(x)=x^m$, $m\ge 1$. Among these last ones, we will also distinguish between those of degree $1$ and those of degree bigger than $1$:

\medskip

\begin{itemize}
  \item [(i)] Any elementary divisor $q(x)=x^m+a_{m-1}x^{m-1}+\dots+a_0$ with $q(0)=a_0\ne 0$ gives rise to an invertible companion matrix
   $$
   C(q(x))=\left(
          \begin{array}{cccc}
            0 & 0 & \dots  & -a_0 \\
            1 & 0 &  &  \vdots\\
             & \ddots & \ddots &  \\
            0 &  & 1 & -a_{m-1}\\
          \end{array}
        \right)\in \mathbb{M}_{m}(\mathbb{F}).
  $$

  \medskip

  \item[(ii)] Any elementary divisor of the form $q(x)=x$ gives rise to the $1\times 1$ companion matrix $ C(x)=(0)$.

  \medskip

  \item[(iii)] Any elementary divisor of the form $q(x)=x^m$, $m>1$, gives rise to a companion  of the form
  $$
   C(x^m)=\left(
          \begin{array}{cccc}
            0 & 0 & \dots  & 0 \\
            1 & 0 &  &  \vdots\\
             & \ddots & \ddots &  \\
            0 &  & 1 & 0\\
          \end{array}
        \right)=\sum_{i=1}^{m-1}e_{i+1,i}\in \mathbb{M}_{m}(\mathbb{F}).
  $$
\end{itemize}

\begin{definition} Let $\mathbb{F}$ be a field, let $n\ge 2$, and let us fix an index of nilpotence $k$ with $2\le k\le n$. For each $1\le r<s\le n$ such that $s+k-2\le n$, we define the following matrices, which will be the ingredients of our main result:
$$
N_{r,s,k}:=e_{r,r}+e_{s,r}-e_{r,s+k-2}-\sum_{i=0}^{k-2} e_{s,s+i}+\sum_{i=0}^{k-3} e_{s+i+1,s+i}\in \mathbb{M}_{n}(\mathbb{F}).
$$
\end{definition}

We begin our work with a series of technicalities, which we need to establish our chief result.

\begin{lemma}\label{Nrsk} Let $\mathbb{F}$ be a field, let $n\ge 2$, and let us fix an index of nilpotence $k$ with $2\le k\le n$. For each $1\le r,s\le n$ such that $s+k-2\le n$, the matrices $N_{r,s,k}$ have rank equal to $k-1$ and are nilpotent of index $k$.
\end{lemma}

\begin{proof} Let $N=\sum_{i=1}^{k-1}e_{i+1,i}\in \mathbb{M}_{n}(\mathbb{F})$ be the  matrix consisting on a single  nilpotent Jordan block of size $k$ and $n-k$ blocks of size 1. By construction, $N$ is nilpotent of index $k$.

We claim that each $N_{r,s,k}$ can be obtained from the  matrix $N$ by an appropriate change of basis. If we denote by $\{e_1, e_2,\dots, e_k,e_{k+1},\dots\}$ the canonical basis,
the matrix
$$N_{1,2,k}=\left(
                  \begin{array}{cccccc|ccc}
                    1 & 0 & 0 & \cdots & 0 & -1 &0&\cdots&0\\
                    1 & -1 & -1 & \cdots & -1 & -1 &0&\cdots&0\\
                    0 & 1 & 0 & 0 & 0 & 0 &0&\cdots&0\\
                    0 & 0 & 1 & 0 & 0 & 0&0&\cdots&0 \\
                    \vdots & \vdots & \vdots & \ddots & \vdots & \vdots& \vdots& \vdots& \vdots \\
                    0 & 0 & 0 & 0 & 1 & 0 &0&\cdots&0\\
\hline
                     0 & 0 & 0 & 0 & 0 & 0 &0&\cdots&0\\
                    \vdots & \vdots & \vdots & \ddots & \vdots & \vdots& \vdots& \vdots& \vdots \\
                     0 & 0 & 0 & 0 & 0 & 0 &0&\cdots&0\\
                  \end{array}
                \right)
$$
is just the same operator $N$ represented on the basis
$$
\mathcal{B}_{1,2,k}=\{v_1=e_1,\, v_2=e_2-e_1,\, \dots,v_k=e_k-e_1,\, v_{k+1}=e_{k+1},\dots\}.
$$
Similarly, if we consider the basis $\mathcal{B}_{r,s,k}=\{v_1,\dots, v_n\}$ where
\begin{align*}
%& v_i=e_{n+1-i},\quad i=1,\dots, r-1 \\
& v_r=e_1\\
& v_{s}=e_2-e_1,\\
& v_{s+1}=e_3-e_1,\\
&\vdots\\
& v_{s+k-2}=e_k-e_1,\\
\end{align*}
and the rest of the vectors $v_i$ of $\mathcal{B}_{r,s,k}$ are any reordering of the vectors $e_{k+1}$, \dots, $e_n$, the matrix $N_{r,s,k}$ is the representation of the operator $N$ on the basis  $\mathcal{B}_{r,s,k}$.

The rank and the index of nilpotence of the matrices $N_{r,s,k}$ is a direct consequence of the rank and the index of nilpotence of the original matrix $N$, as claimed.
\end{proof}

\begin{proposition}\label{invB} Let $\mathbb{F}$ be a field, let $n\ge 2$, and let us fix $k\ge 1$. Also, let $B\in \mathbb{M}_{n}(\mathbb{F})$ be a matrix consisting of a single invertible block of type (i) and size $t$, and also of $r=n-t$ nilpotent blocks of type (ii). If $r\le t(k-1)$ (or, equivalently, ${\rm rank}(B)=t\ge n/k$), then there exists a nilpotent matrix $N_B$ with $N_B^k=0$ such that $B+N_B$ is invertible.
\end{proposition}

\begin{proof}
By hypothesis, the matrix $B$ consists of an invertible block of the form $C(q(x))$, for some polynomial $q(x)$ of degree $t$ with $q(0)\ne 0$, and $r$ nilpotent blocks of type (ii). If $r=0$, we just take $N_B=0$ since $B$ itself is invertible. For the rest of the proof, suppose that $r>0$.

Let us use the classical division theorem to write $r=c(k-1)+d$ with $0\le d< k-1$ ($c$ represents the number of nilpotent matrices of type $N_{r,s,k}$ that we will use in our argument and $d$, if nonzero, means  an extra nilpotent matrix of type $N_{r,s,d+1}$). The condition $r\le t(k-1)$ means that
$$
\left\{
  \begin{array}{ll}
    c\le t, & \hbox{if $d=0$;} \\
    c+1\le t, & \hbox{if $d>0$.}
  \end{array}
\right.
$$

\noindent -- ~ If $c>0$, we consider the matrix
$$
N_{1,t+1,k}+N_{2,t+1+(k-1),k}+\dots+N_{c, t+1+(c-1)(k-1),k}
=\sum_{i=1}^{c} N_{i, t+1+(i-1)(k-1), k}.
$$
By construction, $(\sum_{i=1}^{c} N_{i, t+1+(i-1)(k-1), k})^k=0$ (the matrices $N_{r,s,k}$ appearing in this sum are nilpotent of index $k$ by Lemma \ref{Nrsk} and  they are two-by-two orthogonal).

\noindent -- ~ If $d>0$, we consider $N_{c+1,t+1+c(k-1),d+1}$ which satisfies $(N_{c+1,t+1+c(k-1),d+1})^k=0$ because $d+1<k$; $N_{c+1,t+1+c(k-1),d+1}$ is orthogonal to $\sum_{i=1}^{c} N_{i, t+1+(i-1)(k-1), k}$.

Define the nilpotent matrix
$$
N_B=\underbrace{\sum_{i=1}^{c} N_{i, t+1+(i-1)(k-1), k}}_{\hbox{if $c>0$}}+\underbrace{N_{c+1,t+1+c(k-1),d+1}}_{\hbox{if $d>0$}}, \qquad {N_B^k=0.}
$$

We assert that the matrix $U_B=B+N_B$ is invertible. Indeed, since the determinant of a matrix remains the same if we replace some columns by the original columns to which we add some other columns, \begin{itemize}
\item we add to the first column of $U_B$ the one in position $t+k-1$,
\item we add to the second column of $U_B$ the one in position $t+2k-2$,
\item[] $\vdots$
\item we add to the $c$-column of $U_B$ the one in position $t+c(k-1)$,
\item if $d>0$, we add to the $c+1$-column of $U_B$ the one in position $t+c(k-1)+d$.
\end{itemize}
The condition $c\le t$ if $d=0$ and $c+1\le t$ if $d>0$  assures that these sums of columns in $U_B$ only affects, at most, to the first $t$-columns of $U_B$. We end up with a matrix of the form
$$
\left(
  \begin{array}{c|c|c|c|c}
    C(q(x)) & * & * & * & * \\
\hline
    0 & J_{k-1} & * & * & * \\
\hline
    0 & 0 & J_{k-1} & * & * \\
\hline
   0 & 0 & 0 & \ddots & * \\
\hline
    0 & 0 & 0 & 0 & J_{d} \\
  \end{array}
\right)
$$
where
$$
J_{r}=\left(
          \begin{array}{cccc}
            -1 & -1 & \dots  & -1 \\
            1& 0 & \dots & 0 \\
            0 & \ddots &  & 0 \\
            0 & 0 & 1 & 0 \\
          \end{array}
        \right)\in \mathbb{M}_{r}(\mathbb{F}),\quad  r=k-1 \hbox{ or } d.
$$
Since $\det(J_r)=(-1)^r$, $r=k-1$ or $d$, the determinant of $U_B$ coincides with $\pm$ the determinant of the companion matrix $C(q(x))$, which by hypothesis is nonzero, as required.
\end{proof}

The above proposition can be substantiate by the following concrete construction.

\begin{example}
Let us consider an index of nilpotence $k=5$ and the matrix
$$
B=\left(
    \begin{array}{cc|ccccccc}
      0 & 1 & 0 & 0 & 0 & 0 & 0 & 0 & 0 \\
      1 & 1 & 0 & 0 & 0 & 0 & 0 & 0 & 0 \\
\hline
      0 & 0 & 0 & 0 & 0 & 0 & 0 & 0 & 0 \\
      0 & 0 & 0 & 0 & 0 & 0 & 0 & 0 & 0 \\
      0 & 0 & 0 & 0 & 0 & 0 & 0 & 0 & 0 \\
      0 & 0 & 0 & 0 & 0 & 0 & 0 & 0 & 0 \\
      0 & 0 & 0 & 0 & 0 & 0 & 0 & 0 & 0 \\
      0 & 0 & 0 & 0 & 0 & 0 & 0 & 0 & 0 \\
      0 & 0 & 0 & 0 & 0 & 0 & 0 & 0 & 0 \\
    \end{array}
  \right)\in \mathbb{M}_{9}(\mathbb{F})
$$
consisting on the invertible block $C(x^2-x-1)$ of a degree $t=2$ polynomial and $r=7$ blocks of type (ii). The condition $r\le t(k-1)$ holds. Following the proof of Proposition \ref{invB}, we obtain $c=1$ and $d=3$ in the formula $n-t=c(k-1)+d$; hence we consider the nilpotent matrices
$$
N_{1,3,5}=\left(
    \begin{array}{ccccccccc}
      1 & 0 & 0 & 0 & 0 &-1 & 0 & 0 & 0 \\
      0 & 0 & 0 & 0 & 0 & 0 & 0 & 0 & 0 \\
      1 & 0 &-1 &-1 &-1 &-1 & 0 & 0 & 0 \\
      0 & 0 & 1 & 0 & 0 & 0 & 0 & 0 & 0 \\
      0 & 0 & 0 & 1 & 0 & 0 & 0 & 0 & 0 \\
      0 & 0 & 0 & 0 & 1 & 0 & 0 & 0 & 0 \\
      0 & 0 & 0 & 0 & 0 & 0 & 0 & 0 & 0 \\
      0 & 0 & 0 & 0 & 0 & 0 & 0 & 0 & 0 \\
      0 & 0 & 0 & 0 & 0 & 0 & 0 & 0 & 0 \\
    \end{array}
  \right), %\hbox{ ($N_{1,3,5}^5=0$) }
$$
$$
N_{2,7,4}= \left(
    \begin{array}{ccccccccc}
      0 & 0 & 0 & 0 & 0 &0 & 0 & 0 & 0 \\
      0 & 1 & 0 & 0 & 0 & 0 & 0 & 0 & -1 \\
      0 & 0 &0  &0  &0  &0 & 0 & 0 & 0 \\
      0 & 0 & 0 & 0 & 0 & 0 & 0 & 0 & 0 \\
      0 & 0 & 0 & 0 & 0 & 0 & 0 & 0 & 0 \\
      0 & 0 & 0 & 0 & 0 & 0 & 0 & 0 & 0 \\
      0 & 1 & 0 & 0 & 0 & 0 & -1 & -1 & -1 \\
      0 & 0 & 0 & 0 & 0 & 0 & 1  & 0  & 0 \\
      0 & 0 & 0 & 0 & 0 & 0 & 0  & 1  & 0 \\
    \end{array}
  \right). %\hbox{ ($N_{2,7,4}^4=0$).}
$$
Then $N_B=N_{1,3,5}+N_{2,7,4}$ satisfies $N_B^5=0$; moreover,
$$
B+N_B=\left(
    \begin{array}{cc|cccc|ccc}
      1 & 1 & 0 & 0 & 0 & -1 &0& 0 & 0 \\
      1 & 2 & 0 & 0 & 0 & 0 & 0 & 0 & -1 \\
\hline
      1 & 0 & -1 & -1 & -1 & -1 & 0 & 0 & 0 \\
      0 & 0 & 1 & 0 & 0 & 0 & 0 & 0 & 0 \\
      0 & 0 & 0 & 1 & 0 & 0 & 0 & 0 & 0 \\
      0 & 0 & 0 & 0 & 1 & 0 & 0 & 0 & 0 \\
\hline
      0 & 1 & 0 & 0 & 0 & 0 & -1 & -1 & -1 \\
      0 & 0 & 0 & 0 & 0 & 0 & 1 & 0 & 0 \\
      0 & 0 & 0 & 0 & 0 & 0 & 0 & 1 & 0 \\
    \end{array}
  \right)
$$
is invertible because, if we add the column 6 to column 1 and add column 9 to column 2, it would follow that
$$
\left(
    \begin{array}{cc|cccc|ccc}
      0 & 1 & 0 & 0 & 0 & -1 &0& 0 & 0 \\
      1 & 1 & 0 & 0 & 0 & 0 & 0 & 0 & -1 \\
\hline
      0 & 0 & -1 & -1 & -1 & -1 & 0 & 0 & 0 \\
      0 & 0 & 1 & 0 & 0 & 0 & 0 & 0 & 0 \\
      0 & 0 & 0 & 1 & 0 & 0 & 0 & 0 & 0 \\
      0 & 0 & 0 & 0 & 1 & 0 & 0 & 0 & 0 \\
\hline
      0 & 0 & 0 & 0 & 0 & 0 & -1 & -1 & -1 \\
      0 & 0 & 0 & 0 & 0 & 0 & 1 & 0 & 0 \\
      0 & 0 & 0 & 0 & 0 & 0 & 0 & 1 & 0 \\
    \end{array}
  \right)=\left(
            \begin{array}{c|c|c}
              C(x^2-x-1) & * & * \\
\hline
              0 & J_4 & * \\
\hline
              0 & 0 & J_3 \\
            \end{array}
          \right),
$$
which is clearly invertible, as expected.
\end{example}

\begin{proposition}\label{nilB} Let $\mathbb{F}$ be a field, let $n\ge 2$, and let us fix $k\ge 2$.
Let $C=\sum_{i=1}^{t-1}e_{i+1,i}\in \mathbb{M}_{n}(\mathbb{F})$ be a matrix consisting of a nilpotent block of type (iii) and size $t$, and $r=n-t$ nilpotent blocks of type (ii). If $r\le k-2+(t-1)(k-1)$ (or, equivalently, ${\rm rank}(C)=t-1\ge \frac nk$), then there exists a nilpotent matrix $N_C$ with $N_C^k=0$ and such that $C+N_{C}$ is invertible.
\end{proposition}

\begin{proof} If $r<k-2$,  we take the nilpotent matrix $N_{C}=N_{1,t,r+2}=e_{1,1}+e_{t,1}-e_{1,t+r}-\sum_{i=0}^{r}e_{t,t+i}+\sum_{i=0}^{r-1}e_{t+i+1,t+i}$; then
$$
U_C=C+N_{C}=\sum_{i=1}^{r-1}e_{i+1,i}+e_{1,1}+e_{t,1}-e_{1,t+r}-\sum_{i=0}^{r}e_{t,t+i}.
$$
Adding the column in position $t+r$ to the first column of $U_C$ and replacing row $t$ by the sum of that row and the rest of the rows below, we obtain
the matrix  $C(x^{n}+x^t+1)$, which is the   companion matrix of the polynomial $p(x)=x^{n}+x^t+1$, $p(0)\ne 0$, so $U_C$ is invertible. Moreover, since $r+2< k$, we have $N_C^{r+2}=N_C^k=0$.

If $r\ge k-2$,
arguing as in the proof of Proposition \ref{invB} but beginning at position (2,2), let us use the classical division theorem to write $r-k+2=c(k-1)+d$ with $0\le d< k-1$. Define
\begin{align*}
   N_C&=N_{1,t,k}\\
&+ N_{2,t+(k-1),k}+N_{3,t+2(k-1),k}+\dots+N_{c+1, t+c(k-1),k}\quad \hbox{ if $c>0$}\\
&+N_{c+2,t+(c+1)(k-1),d+1}\quad \hbox{ if $d>0$}\\
&=N_{1,t,k}+\underbrace{\sum_{i=2}^{c+1} N_{i, t+(i-1)(k-1),k}}_{\hbox{ if $c>0$}}+ \underbrace{N_{c+2,t+(c+1)(k-1),d+1}}_{\hbox{ if $d>0$}}.
\end{align*}
The matrix $N_C$ satisfies $N_C^k=0$, because it consists of orthogonal nilpotent matrices of the form $N_{r,s,j}$, $j\le k$, all of them satisfying $N_{r,s,j}^k=0$.

In order to see that $U_C=C+N_C$ is invertible, if
\begin{itemize}
\item we add to the first column of $U_C$ the column in position $t+k-2$,
\item we add to the second column of $U_C$ the one in position $t+2(k-1)-1$,
%  \item we add to the second column of $U_C$ the one in position $t+3(k-1)-1$
\item[] $\vdots$
\item we add to the $c$-column of $U_C$ the one in position $t+(c+1)(k-1)-1$,
\item we add to the $c+1$-column of $U_C$ the one in position $t+(c+1)(k-1)-1+d=n$,
\end{itemize}
and then we replace row $t$ by the sum of that row and  the rows $t+1$,\dots, $t+k-2$ below, we obtain
a matrix of the form
$$
\left(
  \begin{array}{c|c|c|c|c}
    C(x^{t+k-2}+x^t+1) & * & * & * & * \\
\hline
    0 & J_{k-1} & * & * & * \\
\hline
    0 & 0 & J_{k-1} & * & * \\
\hline
   0 & 0 & 0 & \ddots & * \\
\hline
    0 & 0 & 0 & 0 & J_{d} \\
  \end{array}
\right)
$$
where
$$
J_{r}=\left(
          \begin{array}{cccc}
            -1 & -1 & \dots  & -1 \\
            1& 0 & \dots & 0 \\
            0 & \ddots &  & 0 \\
            0 & 0 & 1 & 0 \\
          \end{array}
        \right)\in \mathbb{M}_{r}(\mathbb{F}),\quad  r=k-1 \hbox{ or } d.
$$
Since $\det(J_r)=(-1)^r$, $r=k-1$ or $d$, the determinant of $U_C$ coincides with $\pm$ the determinant of the companion matrix $C(x^{t+k-2}+x^t+1)$, which is nonzero, as needed.
\end{proof}

The next concrete construction will materialize the last proposition.

\begin{example}
Let us consider an index of nilpotence $k=4$ and the nilpotent matrix
$$
C=\left(
    \begin{array}{cccc|ccccc}
      0 & 0 & 0 & 0 & 0 & 0 & 0 & 0 & 0 \\
     1 & 0 & 0 & 0 & 0 & 0 & 0 & 0 & 0 \\
      0 & 1 & 0 & 0 & 0 & 0 & 0 & 0 & 0 \\
      0 & 0 & 1 & 0 & 0 & 0 & 0 & 0 & 0 \\
\hline
      0 & 0 & 0 & 0 & 0 & 0 & 0 & 0 & 0 \\
      0 & 0 & 0 & 0 & 0 & 0 & 0 & 0 & 0 \\
      0 & 0 & 0 & 0 & 0 & 0 & 0 & 0 & 0 \\
      0 & 0 & 0 & 0 & 0 & 0 & 0 & 0 & 0 \\
      0 & 0 & 0 & 0 & 0 & 0 & 0 & 0 & 0 \\
    \end{array}
  \right)\in \mathbb{M}_{9}(\mathbb{F})
$$
consisting on a nilpotent block of size $t=4$ and $r=5$ blocks of type (ii). Since $r\ge k-2$, imitating the proof of Proposition \ref{nilB} we first consider
$$
N_{1,4,4}=\left(
    \begin{array}{cccccc|ccc}
      1 & 0 & 0 & 0 & 0 & -1 & 0 & 0 & 0 \\
     0 & 0 & 0 & 0 & 0 & 0 & 0 & 0 & 0 \\
      0 & 0 & 0 & 0 & 0 & 0 & 0 & 0 & 0 \\
     1 & 0 & 0 & -1 & -1& -1& 0 & 0 & 0 \\
      0 & 0 & 0 & 1 & 0 & 0 & 0 & 0 & 0 \\
      0 & 0 & 0 & 0 & 1 & 0 & 0 & 0 & 0 \\
\hline
      0 & 0 & 0 & 0 & 0 & 0 & 0 & 0 & 0 \\
      0 & 0 & 0 & 0 & 0 & 0 & 0 & 0 & 0 \\
      0 & 0 & 0 & 0 & 0 & 0 & 0 & 0 & 0 \\
    \end{array}
  \right).
$$
Moreover, since $r=5$ and $k=4$, we can get $c=1$ and $d=0$ in the formula $r-k+2=c(k-1)+d$, so we also consider the matrix
$$
N_{2,7,4}=\left(
    \begin{array}{ccccccccc}
      0 & 0 & 0 & 0 & 0 & 0 & 0 & 0 & 0 \\
     0 & 1 & 0 & 0 & 0 & 0 & 0 & 0 & -1 \\
      0 &0 & 0 & 0 & 0 & 0 & 0 & 0 & 0 \\
      0 & 0 & 0 & 0 & 0 & 0 & 0 & 0 & 0 \\
      0 & 0 & 0 & 0 & 0 & 0 & 0 & 0 & 0 \\
      0 & 0 & 0 & 0 & 0 & 0 & 0 & 0 & 0 \\
      0 & 1 & 0 & 0 & 0 & 0 & -1 & -1 & -1 \\
      0 & 0 & 0 & 0 & 0 & 0 & 1 & 0 & 0 \\
      0 & 0 & 0 & 0 & 0 & 0 & 0 & 1 & 0 \\
    \end{array}
  \right).
$$
Thus $N_C=N_{1,4,4}+N_{2,7,4}$ satisfies $N_C^4=0$ and
$$
C+N_{C}=\left(
    \begin{array}{cccccc|ccc}
      1 & 0 & 0 & 0 & 0 & -1 & 0 & 0 & 0 \\
     1 & 1 & 0 & 0 & 0 & 0 & 0 & 0 & -1 \\
      0 &1 & 0 & 0 & 0 & 0 & 0 & 0 & 0 \\
      1 & 0 & 1 & -1 & -1 & -1 & 0 & 0 & 0 \\
      0 & 0 & 0 & 1 & 0 & 0 & 0 & 0 & 0 \\
      0 & 0 & 0 & 0 & 1 & 0 & 0 & 0 & 0 \\
\hline
      0 & 1 & 0 & 0 & 0 & 0 & -1 & -1 & -1 \\
      0 & 0 & 0 & 0 & 0 & 0 & 1 & 0 & 0 \\
      0 & 0 & 0 & 0 & 0 & 0 & 0 & 1 & 0 \\
    \end{array}
  \right)
$$
is invertible because, if we add the column 4 to column 1, add column 9 to column 2 and add rows 5 and 6 to row 4, we will receive
$$
\left(
    \begin{array}{ccc ccc|ccc}
      0 & 0 & 0 & 0 & 0 & -1 & 0 & 0 & 0 \\
     1 & 0& 0 & 0 & 0 & 0 & 0 & 0 & -1 \\
      0 &1 & 0 & 0 & 0 & 0 & 0 & 0 & 0 \\

      0 & 0 & 1 & 0 & 0 & -1 & 0 & 0 & 0 \\
      0 & 0 & 0 & 1 & 0 & 0 & 0 & 0 & 0 \\
      0 & 0 & 0 & 0 & 1 & 0 & 0 & 0 & 0 \\
\hline
      0 & 0 & 0 & 0 & 0 & 0 &  -1 & -1 & -1 \\
      0 & 0 & 0 & 0 & 0 & 0 & 1 & 0 & 0 \\
      0 & 0 & 0 & 0 & 0 & 0 & 0 & 1 & 0 \\
    \end{array}
  \right)=\left(
            \begin{array}{c|c}
              C(x^6+x^3+1) & * \\
\hline
              0 & J_3 \\
            \end{array}
          \right),
$$
which is clearly invertible, as promised.
\end{example}

Combining the previous two propositions we reach the main result which motivates writing of this article.

\begin{theorem}\label{main} Let $\mathbb{F}$ be a field, let $n\ge 2$, and let us fix $k\ge 1$. Given a nonzero matrix $A\in \mathbb{M}_{n}(\mathbb{F})$,  there exists an invertible matrix $U\in \mathbb{M}_{n}(\mathbb{F})$ and a nilpotent matrix $N\in \mathbb{M}_{n}(\mathbb{F})$ with $N^k=0$ such that $A=U+N$ if, and only if, the rank of $A$ is greater than or equal to $\frac{n}{k}$.
\end{theorem}

\begin{proof} As already mentioned in Remark \ref{necessarycondition}, a necessary condition to express $A$ as the sum $U+N$, where $U$ is invertible and $N^k=0$, is that the rank of $A$ is greater than or equal to $\frac nk$ because $n={\rm rank}(U)={\rm rank} (A-N)\le {\rm rank} (A)+{\rm rank}(N)\le {\rm rank} (A)+n-\frac nk$.

Conversely, suppose that the rank of $A\in \mathbb{M}_{n}(\mathbb{F})$ is no less than $\frac nk$. Without loss of generality, we may assume that $A$ is expressed in its primary rational canonical form, i.e., it is a sum of the companion matrices  of its elementary divisors. Let us show  that there will exist an invertible matrix $U$ and a nilpotent matrix $N$, $N^k=0$, such that $A=U+N$, as pursuing.

Let us reorder the blocks of matrix $A$ -- which just corresponds to a reordering of the elementary divisors of $A$ -- as follows:
\begin{itemize}
  \item [(1)] we follow each $t_i\times t_i$ block $B_i$ of type (i) by $s_i$   blocks of type (ii) ($s_{i}\le t_i(k-1)$),

  \medskip

  \item [(2)] we follow each $t_j\times t_j$ block $C_j$ of type (iii) by $r_j$   blocks of type (ii) ($r_{j}\le k-2+(t_j-1)(k-1)$).
\end{itemize}
such that $\sum_{i} s_i+\sum_j r_j=n-{\rm rank}(A)$.
The rank of $A$ guarantees that all blocks of type (ii) can be distributed and combined with either blocks of type (i) or with blocks of type (iii) by following points (1) and (2).

We, thereby, have the following two cases:

$\bullet$ In accordance with Proposition \ref{invB}, for every invertible block $B_i$ of rank $t_i$ which is followed by $s_i$ blocks of type (ii), $s_i\le t_i(k-1)$, there exists a nilpotent matrix $N_{B_i}$ such that $N_{B_i}^k=0$ such that $B_i+N_{B_i}$ is an  invertible matrix of rank $t_i+s_i$.

$\bullet$ In accordance with Proposition \ref{nilB}, for every nilpotent  block $C_j$ of index $t_j$ and rank $t_j-1$ which is followed by $r_j$ blocks of type (ii),  $r_j\le k-2+(t_j-2)(k-1)$, there exists a nilpotent matrix $N_{C_j}$ such that $N_{C_j}^k=0$ such that $C_j+N_{C_j}$ is an  invertible matrix of  rank $t_j+r_j$.

%Moreover, if at any time we have no more blocks of type (ii), by Lemma \ref{nilB} we add to each nilpotent Jordan block  of index $s_i$ and rank $s_i-1$ the matrix $N_{1,t-1,2}$ to obtain an invertible matrix.
Now, define $N=-\sum_i N_{B_i}-\sum_j N_{C_j}$. Since the nilpotent matrices that we add are mutually orthogonal, we therefore can get a nilpotent matrix $N$ with $N^k=0$ and such that $U=A-N$ is invertible:
$$
A-N=\left(
  \begin{array}{c|c|c|c|c|c}
    B_1+N_{B_1} & 0 & 0 & 0 & 0 & 0 \\
\hline
    0 &  B_2+N_{B_2} & 0 & 0 & 0 & 0 \\
\hline
    0 & 0 & \ddots & 0 & 0 & 0 \\
\hline
    0 & 0 & 0 &  C_1+N_{C_1} & 0 & 0 \\
\hline
    0 & 0 & 0 & 0 &  C_2+N_{C_2} & 0 \\
\hline
    0 & 0 & 0 & 0 & 0 & \ddots \\
  \end{array}
\right).
$$
Finally, we decompose $A=U+N$, as stated.
\end{proof}

In conclusion, it is worthwhile noticing that the key tool in our arguments is the primary rational canonical form of any square matrix, which holds for matrices over arbitrary fields. However, since the mentioned above Calugareanu-Lam's result from \cite{CL} about the decomposition of matrices into invertible and nilpotent is true for matrices over division rings \cite[Remark 3.12]{CL}, we can close our work by posing the following query:

\medskip

\noindent{\bf Open Problem:} Given a fixed bound $k\ge 1$ for the index of nilpotence, find necessary and sufficient conditions to expressed every nonzero square matrix over a division ring as the sum of an invertible matrix and a nilpotent matrix $N$ with $N^k=0$.

\bigskip
\bigskip
\bigskip

\noindent{\bf Funding:} The first-named author (Peter V. Danchev) was supported in part by the Bulgarian National Science Fund under Grant KP-06 No. 32/1 of December 07, 2019, the second-name author (Esther Garc\'{\i}a) was partially supported by Ayuda Puente 2022, URJC. The three authors were partially supported by the Junta de Andaluc\'{\i}a FQM264.

\vskip3.0pc

\bibliographystyle{plain}

\end{document}